\documentclass[reqno]{amsart}

\usepackage{etoolbox}
\patchcmd{\section}{\scshape}{\bfseries}{}{}
\patchcmd{\subsection}{-.5em}{.5em}{}{}
\makeatletter
\renewcommand{\@secnumfont}{\bfseries}
\makeatother

\usepackage{amsfonts}
\usepackage{amssymb}
\usepackage{enumerate} 
\usepackage[lite]{amsrefs}

\usepackage{amsthm}
\usepackage{amsmath}
\usepackage{hyperref}
\usepackage[left=37 mm, right=37 mm]{geometry}

\title{Formalizing Elements of Probabilistic Mechanics}
\date{\today}


\author{Farida Kachapova}
\address{Auckland University of Technology
\\Auckland, New Zealand}
\email{farida.kachapova@aut.ac.nz} 
\author{Ilias Kachapov}
\address{University of Auckland\\ Auckland, New Zealand}
\email{kachapov@xtra.co.nz} 
\newpage

\newtheorem{theorem}{Theorem}[section]

\newtheorem{lemma}[theorem]{Lemma}
\newtheorem{proposition}[theorem]{Proposition}
\newtheorem{cor}[theorem]{Corollary}
\newtheorem{definition}[theorem]{Definition}
\newtheorem{notation}[theorem]{Notation}

\begin{document}

\maketitle

\section*{Abstract}
In this paper we create a model of particle motion on a three-dimensional lattice using discrete random walk with small steps. We rigorously construct a probability space of the particle trajectories. Unlike deterministic approach in classical mechanics, here we use probabilistic properties of particle movement to formally derive analogues of Newton's first  and second laws of motion. Similar probabilistic models can potentially be applied to justify laws of thermodynamics in a consistent manner. 

\section{Introduction}
This paper is motivated by some unsolved problems in  thermodynamics. Heat transfer is regarded  as the transfer of kinetic energy from one particle to another when they collide. This process is time-reversible: if we change the direction of time (or the direction of velocity), then kinetic energy is transferred from the second particle to the first one. Thus, heat transfer is considered a deterministic process (when initial conditions uniquely determine the future and past positions, velocity and other characteristics), therefore it follows laws of classical mechanics. 

However, this is inconsistent with the fact that heat transfer at the macroscopic level is not time-reversible: heat can be transferred from a hot solid to a cold one but not back. Another example: when  hot and cold liquids are mixed the result is a warm liquid, but the warm liquid cannot be separated into the hot and cold liquids. 

This inconsistency might be resolved if we assume that molecule movement follows probabilistic rather than deterministic laws because probabilistic processes are irreversible. At the same time such a theory need to be consistent with Newton's laws of motion and the rest of the classical mechanics at the macroscopic level. Research in this direction is being conducted such as the study of relations between microscopic and global properties of physical objects in statistical mechanics  (see for example, \cite{S} and \cite{DS}). 

In this paper we contribute in this direction of research by creating a model where a particle moves randomly in small steps while its long distance movement becomes approximately deterministic.
We use a rigorous approach in construction of this model. 
In Section 2 we introduce a probability space $(\Omega,\mathcal{F},P)$ of particle trajectories on a lattice with small cells. 
Probability measure for this space is defined by applying Hahn–Kolmogorov extension theorem (see \cite{N}, \cite{D}, \cite{R}). 

In Section 3 we introduce a stochastic process - a random walk $\mathbf{R}_n$ on the lattice using the probability space $(\Omega,\mathcal{F},P)$, and derive properties of $\mathbf{R}_n$. $\mathbf{R}_n$ represents the position of the particle at time $n\tau$, where $\tau$ is a fixed small unit of time. Velocity and acceleration of the particle are defined in terms of the random walk.

In Section 4 we state an axiom about probabilities of particle motion under zero resultant force and derive an analogue of Newton's first law of motion. 

In Section 5 we state an axiom about probabilities of particle motion under a non-zero resultant force and derive an analogue of Newton's second law of motion, plus a couple of related facts. 

\section{Probability space} 

In this section we construct a probability space $(\Omega,\mathcal{F},P)$.
We will use the following concepts from measure theory.

\begin{definition}
Suppose $\Omega$ is a non-empty set and 
$\mathcal{A}$ is a family of subsets of $\Omega$. 

$\mathcal{A}$ is called an \textbf{algebra over} $\Omega$ if it has the properties:
\begin{enumerate}
\item $\varnothing\in\mathcal{A}$,
\smallskip
\item $B\in\mathcal{A}\;\Rightarrow\; B^c=\Omega\setminus B\in\mathcal{A}$,
\smallskip
\item $B,C\in\mathcal{A}\;\Rightarrow\; B\cup C\in\mathcal{A}$.
\end{enumerate}
\end{definition}

\begin{definition}
Suppose $\Omega$ is a non-empty set and 
$\mathcal{F}$ is a family of subsets of $\Omega$. 

$\mathcal{F}$ is called a $\boldsymbol{\sigma}$\textbf{-algebra on} $\Omega$ if it has the properties:
\begin{enumerate}
\item $\varnothing\in\mathcal{F}$,
\smallskip
\item $B\in\mathcal{F}\;\Rightarrow\; B^c=\Omega\setminus B\in\mathcal{F}$,
\smallskip
\item for any elements $B_1,B_2,B_3,\ldots$ of $\mathcal{F}$: 
$\quad\bigcup\limits_{n=1}^\infty B_n\in\mathcal{F}$.
\end{enumerate}
\end{definition}

\subsection{Sample space and sigma-algebra of events}

\begin{definition}
1) Denote $\mathbb{N}_7=\{0,1,2,3,4,5,6\}$.
For the \textbf{sample space} we choose the set $\Omega =(\mathbb{N}_7)^\mathbb{N}$. 

Thus, any element $\omega$ of $\Omega$ is an infinite sequence $\omega=
(\omega_0,\omega_1,\omega_2, \ldots)$ of natural numbers between 0 and 6.
 Here $\omega_k$ denotes the $k$-th element of the sequence $\omega$ $(k=0,1,2,\ldots).$

Such a sequence represents \textbf{trajectory of a particle} on 3-dimensional lattice (details are given in Section 3). 

2) A \textbf{hyperplane} is defined as a set of the form:
\[\boldsymbol{C(n,i)}=\{\omega\in\Omega: w_n=i\},\]
where $n\in\mathbb{N}$ and $i\in\mathbb{N}_7$.
\end{definition}

For each $n\in\mathbb{N}$ there are seven distinct hyperplanes: $C(n,0),C(n,1),C(n,2),C(n,3)$, $C(n,4),C(n,5),C(n,6)$.

\begin{proposition}
1) If $i\neq j$, then $C(n,i)\cap C(n,j)=\varnothing$.
\medskip

2) For any $n\in\mathbb{N}$:
\[\bigcup_{j=0}^6 C(n,j)=\Omega.\]

3) Complement of any hyperplane $C(n,i)$ is a union of the other six hyperplanes:
$$\big(C(n,i)\big)^c=\bigcup_{\substack{j=0\\j\neq i}}^6 C(n,j).$$
\label{prop:hyperplane}
\end{proposition}

Proof of the proposition is obvious.

\begin{definition}
1) Denote $\boldsymbol{\mathcal{F}_0}$ the algebra generated by all hyperplanes $C(n,i)$, where $n\in\mathbb{N}$ and $i\in\mathbb{N}_7$. This means that $\mathcal{F}_0$ is the least algebra over $\Omega$ containing all hyperplanes.
\medskip

2) Denote $\boldsymbol{\mathcal{F}}$ the $\sigma$-algebra generated by $\mathcal{F}_0$, i.e. the least $\sigma$-algebra on $\Omega$ containing $\mathcal{F}_0$. Elements of $\mathcal{F}$ are called \textbf{events}.
\end{definition}

Thus, we have introduced $\Omega$ and $\mathcal{F}$. To complete the definition of the probability space $(\Omega,\mathcal{F},P)$, we need to construct a probability measure $P$.

\subsection{Constructing probability measure on $(\Omega,\mathcal{F})$} 

\subsubsection{Planes} 
\begin{definition}
We define a \textbf{plane} as a subset of $\Omega$ that can be represented as the following intersection of hyperplanes:
\[\boldsymbol{D[i_0,i_1,\ldots,i_n]} =C(0,i_0)\cap C(1,i_1)\cap\ldots \cap C(n,i_n),\]
where $n\in\mathbb{N}$ and each $i_k\in\mathbb{N}_7$.
\end{definition}

In particular, $D[i] =C(0,i)$.
For each $n\in\mathbb{N}$ there are $7^{n+1}$ planes $D[i_0,i_1,\ldots,i_n]$.

\begin{proposition}
1) If $D[i_0,i_1,\ldots,i_n]\cap D[j_0,j_1,\ldots,j_n]\neq\varnothing$, 
\[\text{then }i_0=j_0,i_1=j_1,\ldots,i_n=j_n\text{ and }
D[i_0,i_1,\ldots,i_n]= D[j_0,j_1,\ldots,j_n].\]

2) If $m<n$ and $D[i_0,i_1,\ldots,i_n]\cap D[j_0,j_1,\ldots,j_m]\neq\varnothing$,
\[\text{then }i_0=j_0,i_1=j_1,\ldots,i_m=j_m\text{ and }D[i_0,i_1,\ldots,i_n]\subseteq D[j_0,j_1,\ldots,j_m].\]

3) If planes $D_1$ and $D_2$ intersect, then $D_1\subseteq D_2$ or $D_2\subseteq D_1$.
\medskip

4) Intersection of two planes is $\varnothing$ or a plane.
\label{prop:planes}
\end{proposition}
\begin{proof}
1) Suppose $D[i_0,i_1,\ldots,i_n]\cap D[j_0,j_1,\ldots,j_n]\neq\varnothing$. Then
\[\left[\bigcap\limits_{k=0}^nC(k,i_k)\right]\cap\left[\bigcap\limits_{k=0}^nC(k,j_k)\right]\neq\varnothing\]
and there exists $\omega\in\Omega$ such that for any $k=0,1,2,\ldots,n,$ $i_k=\omega_k=j_k$.

So $D[i_0,i_1,\ldots,i_n]= D[j_0,j_1,\ldots,j_n].$
\medskip

2) Suppose $m<n$ and $D[i_0,i_1,\ldots,i_n]\cap D[j_0,j_1,\ldots,j_m]\neq\varnothing$. Then
\[\left[\bigcap\limits_{k=0}^nC(k,i_k)\right]\cap\left[\bigcap\limits_{k=0}^mC(k,j_k)\right]\neq\varnothing\]
and some $\tilde{\omega}$ belongs to this intersection. Then
\begin{equation}
\text{for any }k=0,1,2,\ldots,m,\quad i_k=\tilde{\omega}_k=j_k.
\label{eq:index_equality}
\end{equation}

Suppose $\omega\in D[i_0,i_1,\ldots,i_n]$. Then by \eqref{eq:index_equality},  for any $k=0,1,2,\ldots,m,\quad\omega_k=i_k=j_k$. So 
$\omega\in D[j_0,j_1,\ldots,j_m]$.
\medskip

This proves $D[i_0,i_1,\ldots,i_n]\subseteq D[j_0,j_1,\ldots,j_m]$.
\medskip

3) This follows from parts 1) and 2).
\medskip

4) This follows from part 3).
\end{proof}

\begin{proposition} 
1) Any element $A$ of $\mathcal{F}_0$ is $\varnothing$ or can be represented as a finite union of planes. 

2) Any element $A$ of $\mathcal{F}_0$ is $\varnothing$ or can be represented as a finite union of pairwise disjoint planes. 
\label{prop:union_planes}
\end{proposition}
\begin{proof}
1) By the definition of algebra $\mathcal{F}_0$ its elements are constructed in the following steps.
\begin{itemize}
\item $\varnothing\in\mathcal{F}_0$.
\smallskip

\item Any hyperplane $C(n,i)\in\mathcal{F}_0$ $\;(n\in\mathbb{N},i\in\mathbb{N}_7)$.
\smallskip

\item If $B\in\mathcal{F}_0$, then $B^c\in\mathcal{F}_0$.
\smallskip

\item If $B,H\in\mathcal{F}_0$, then $B\cup H\in\mathcal{F}_0$.
\end{itemize}

Consider an arbitrary $A\in\mathcal{F}_0$. We will prove the statement by induction on construction of $A$.

a) Case $A=\varnothing$ is obvious.
\medskip

b) Case $A=C(n,i)$. By Proposition \ref{prop:hyperplane}.2), 
\[\Omega=\bigcup_{j=0}^6 C(n,j).\]
\[A=\underbrace{
\Omega\cap\Omega\ldots\cap\Omega}_{n\text{ times}}\cap C(n,i)
=\left[\bigcup_{j_0=0}^6 C(0,j_0)\right]\cap\ldots\cap\left[\bigcup_{j_{n-1}=0}^6 C(n-1,j_{n-1})\right]\cap C(n,i)\]
\[=\bigcup\limits_{j_0=0}^6\ldots\bigcup\limits_{j_{n-1}=0}^6\left[C(0,j_0)\cap\ldots\cap C(n-1,j_{n-1})\cap C(n,i)\right]\]
\[=\bigcup\limits_{j_0=0}^6\ldots\bigcup\limits_{j_{n-1}=0}^6D[j_0,\ldots,j_{n-1},i]\text{ is a finite union of planes}.\]

c) First we consider complement of a plane. 
\[\left(D[i_0,\ldots,i_n]\right)^c=\bigcup\limits_{k=0}^n \left(C(k,i_k)\right)^c
=\bigcup\limits_{k=0}^n\bigcup\limits_{\substack{j_k=0\\j_k\neq i_k}}^6 C(k,j_k)\]
by Proposition \ref{prop:hyperplane}.3).
We proved in part b) that any hyperplane is a finite union of planes. So 
\begin{equation}
\left(D[i_0,\ldots,i_n]\right)^c \text{ is a finite union of planes}.
\label{eq:complement}
\end{equation}

Now we consider the case $A=B^c$, where $B\in \mathcal{F}_0$.

If $B=\varnothing$, then by Proposition \ref{prop:hyperplane}.2), $A=\Omega=\bigcup\limits_{i=0}^6 C(0,i)=\bigcup\limits_{i=0}^6 D[i]$ is a finite union of planes.

Suppose $B\neq\varnothing$. Then by the inductive assumption $B$ is a finite union of planes:
\[B=\bigcup\limits_{k=0}^s D_k\text{ and }A=\bigcap\limits_{k=0}^s D_k^c.\]

By \eqref{eq:complement} each $D_k^c$ is  a finite union of planes:
\[D_k^c=\bigcup\limits_{l_k=0}^{r_k} D'_{l_k}.\]
\[\text{So }A=\bigcap\limits_{k=0}^s\bigcup\limits_{l_k=0}^{r_k} D'_{l_k}=\bigcup\limits_{l_1=0}^{r_1}\ldots\bigcup\limits_{l_s=0}^{r_s}\left(D'_{l_1}\cap\ldots\cap D'_{l_s}\right).\]
By Proposition \ref{prop:planes}.4), each 
$D'_{l_1}\cap\ldots\cap D'_{l_s}$ is $\varnothing$ or a plane. Then $A$ is $\varnothing$ or a finite union of planes.

d) Case $A=B\cup H$, where $B,H\in\mathcal{F}_0$. 

By the inductive assumption $B$ is a finite union of planes and so is $H$. Clearly, this is also true for $A$.
\medskip

2) Suppose $A\neq\varnothing$. Then by part 1), $A=\bigcup\limits_{k=0}^nD_k$, where each $D_k$ is a plane.
We will prove the statement by induction on $n$.

a) Case $n=1$ is obvious. 

b) Assume the statement holds for $n$ and $A=\bigcup\limits_{k=0}^{n+1}D_k$. Denote $B=\bigcup\limits_{k=0}^nD_k$. By the inductive assumption $B$ can be represented as a finite union of pairwise disjoint planes:
\begin{equation}
B=\bigcup\limits_{l=0}^mD'_l.
\label{eq:pairwise_disjoint}
\end{equation}
\begin{equation}
\text{So }A=\bigcup\limits_{l=0}^mD'_l\cup D_{n+1}.
\label{eq:disjoint}
\end{equation}

Denote $D'_{l_1},\ldots,D'_{l_r}$ all the planes from the union \eqref{eq:pairwise_disjoint}
that intersect with $D_{n+1}$. (If there are no such planes, then the statement is proven.)

If $D_{n+1}\subseteq D'_{l_i}$ for some $i=1,2,\ldots,r$, then $A=\bigcup\limits_{l=0}^mD'_l$ and the statement is proven. Otherwise each $ D'_{l_i}\subseteq D_{n+1}$ by Proposition \ref{prop:planes}.3) and we can remove each $D'_{l_i}$ from the union \eqref{eq:disjoint}:
\[A=\bigcup\limits_{\substack{l=0,\\l\neq l_1,\\\ldots,\\\l\neq l_r}}^mD'_l\cup D_{n+1};\]
getting a finite union of pairwise disjoint planes.
\end{proof}

\subsubsection{Defining $P_0$ on $\mathcal{F}_0$} 

\begin{notation}
For each $n\in\mathbb{N}$ we fix seven non-negative numbers $p_n(i),i=0,1,\ldots,6$, such that $\sum\limits_{i=0}^6 p_n(i)=1$. 
We consider these numbers fixed for the rest of the article.
\end{notation}

A simple example of such numbers is when all $p_n(i)=\dfrac{1}{7}$.

\begin{definition}
\begin{enumerate}
\item $P_0(\varnothing) =0$.
\item $P_0(C(n,i))=p_n(i)$ $\;(n\in\mathbb{N},i\in\mathbb{N}_7)$.
\medskip
\item For any plane $
D[i_0,i_1,\ldots,i_n] =C(0,i_0)\cap C(1,i_1)\cap\ldots \cap C(n,i_n)$:
\[P_0\big(D[i_0,i_1,\ldots,i_n]\big)=\prod\limits_{k=0}^n P_0\big(C(k,i_k)\big).\]
\end{enumerate}
\label{def:P_0}
\end{definition}

\begin{lemma}
If a plane $D$ equals a union of pairwise disjoint planes:
\[D=\bigcup\limits_{j=0}^k D_j,\]
then $P_0(D)=\sum\limits_{j=0}^k P_0(D_j)$.
\label{lemma:hard_proof}
\end{lemma}

\begin{proof}
For any plane $X$ we define its length by the following:
\[\text{if }X=D[j_0,j_1,\ldots,j_m],\text{ then }\boldsymbol{length}(X)=m.\]

Denote $D=D[i_0,i_1,\ldots,i_n].$ So $length(D)=n$. 

Suppose 
\begin{equation}
D=\bigcup\limits_{j=0}^k D_j.
\label{eq:union}
\end{equation}

Denote $M=\{D_j\,:\,j=0,1,\ldots,k\}.$ 
Next we will prove for any $X\in M$:
\begin{equation}
length(X)\geqslant n\text{ and }X=D[i_0,i_1,\ldots,i_n,i_{n+1},\ldots,i_m]\text{ for some }i_{n+1},\ldots,i_m.
\label{eq:X_form}
\end{equation}
\\
\textit{Proof of }\eqref{eq:X_form}. Any $X\in M$ has the form: $X=D[j_0,\ldots,j_m]$. $length(X)=m$.
\medskip

Assume $m<n$. By \eqref{eq:union}, 
\[X\subseteq D,\]
so $X\cap D\neq\varnothing$. By Proposition \ref{prop:planes}.2), $D[i_0,i_1,\ldots,i_n]\subseteq D[j_0,j_1,\ldots,j_m]$, that is $D\subseteq X$. 
\medskip

So $X=D$ and $m=n$, which contradicts the assumption $m<n$. $m\geqslant n$ is proven.

Thus, $X$ has the form: $X=D[j_0,j_1,\ldots,j_n,j_{n+1},\ldots,j_m]$. Since $X\subseteq D$, by Proposition \ref{prop:planes}.1), 2), $j_0=i_0,j_1=i_1,\ldots,j_n=i_n$. This completes the proof of \eqref{eq:X_form}.
\medskip

Denote $q=\max\{length(X)\,:\,X\in M\}$. By \eqref{eq:X_form}, $q\geqslant n$. We will prove:
\begin{equation}
P_0(D)=\sum\limits_{j=0}^k P_0(D_j)
\label{eq:sum_P_0}
\end{equation}
by induction on $q$.

\textit{Base step}: $q=n$.
In this case any $X\in M$ has the form $X=D[i_0,i_1,\ldots,i_n] =D$. Since elements of $M$ are pairwise disjoint, $D$ is the only element of $M$. That makes equality \eqref{eq:sum_P_0} obvious.

\textit{Inductive step}. Assume \eqref{eq:sum_P_0} holds for $q$. Let $\max\{length(X)\,:\,X\in M\}=q+1$. Denote:
\[M'=\{X\in M\,:\,length(X)\leqslant q\}\text{ and } M''=\{X\in M\,:\,length(X)=q+1\}.\]
Then $M=M'\cup M'',\;M''\neq\varnothing$.
Denote 
\[I=\{(i_{n+1},\ldots,i_q)\,:\,D[i_0,i_1,\ldots,i_n,i_{n+1},\ldots,i_q,l]\in M''\text{ for some }l\}.\] 
Fix $(i_{n+1},\ldots,i_q)\in I$. We will prove:
\begin{equation}
\text{for any }l\in\mathbb{N}_7,\,D[i_0,i_1,\ldots,i_n,i_{n+1},\ldots,i_q,l]\in M''.
\label{eq:index_set}
\end{equation}

By the definition of $I$, for some $l_0\in \mathbb{N}_7$, $X=D[i_0,i_1,\ldots,i_n,i_{n+1},\ldots,i_q,l_0]\in M''$. Without loss of generality we can assume $l_0=6$ and prove \eqref{eq:index_set} for $l=5$ (the general case is similar).

Take $\omega=(i_0,i_1,\ldots,i_n,i_{n+1},\ldots,i_q,5,0,0,\ldots)$. Then $\omega\in D$. By \eqref{eq:union}, $\omega\in Y$ for some $Y\in M=M'\cup M''$. Two cases.
\medskip

\textit{Case 1}: $Y\in M'$.
Then $Y\neq X$ and $Y=D[i_0,\ldots,i_m]$, where $n\leqslant m\leqslant q$. 
\medskip

For $\omega'=(i_0,i_1,\ldots,i_n,i_{n+1},\ldots,i_q,6,0,0,\ldots)$ we have: $\omega'\in X\cap Y$. This contradicts the fact that elements of $M$ are pairwise disjoint. Case 1 is impossible.
\medskip

\textit{Case 2}: $Y\in M''$. Since $\omega\in Y$, we have $Y=D[i_0,i_1,\ldots,i_n,i_{n+1},\ldots,i_q,5]$, that is 
\medskip
\\
$D[i_0,i_1,\ldots,i_n,i_{n+1},\ldots,i_q,5]\in M''$. This completes the proof of \eqref{eq:index_set}.

\eqref{eq:index_set} implies the following:
\[M''=\{D[i_0,i_1,\ldots,i_n,i_{n+1},\ldots,i_q,l]\,:\,(i_n,i_{n+1},\ldots,i_q)\in I\text{ and }l\in\mathbb{N}_7\}.\]

Since $P_0\big(D[i_0,i_1,\ldots,i_n,i_{n+1},\ldots,i_q,l]\big)
=P_0\big(D[i_0,i_1,\ldots,i_n,i_{n+1},\ldots,i_q]\big)\cdot
 P_0\big(C(q+1,l)\big)$, we have:
\[\sum_{X\in M''}P_0(X)
=\sum_{(i_n,i_{n+1},\ldots,i_q)\in I}\sum_{l=0}^6P_0\big(D[i_0,i_1,\ldots,i_n,i_{n+1},\ldots,i_q,l]\big)\]
\[=\sum_{(i_n,i_{n+1},\ldots,i_q)\in I}P_0\big(D[i_0,i_1,\ldots,i_n,i_{n+1},\ldots,i_q]\big)\cdot
\sum_{l=0}^6P_0\big(C(q+1,l)\big).\] 
Since the last sum equals 1, we get:
\begin{equation}
\sum_{X\in M''}P_0(X)
=\sum_{(i_n,i_{n+1},\ldots,i_q)\in I}P_0\big(D[i_0,i_1,\ldots,i_n,i_{n+1},\ldots,i_q]\big).
\label{eq:P_0_X}
\end{equation}
\[\bigcup_{l=0}^6D[i_0,i_1,\ldots,i_n,i_{n+1},\ldots,i_q,l]
=D[i_0,i_1,\ldots,i_n,i_{n+1},\ldots,i_q]\cap\left[\bigcup_{l=0}^6
C(q+1,l)\right]\]
\[=D[i_0,i_1,\ldots,i_n,i_{n+1},\ldots,i_q],\]
since the union in square brackets equals $\Omega$ by Proposition 
\ref{prop:hyperplane}.2). This implies:
\[\bigcup\{X\,:\,X\in M''\}=
\bigcup_{(i_n,i_{n+1},\ldots,i_q)\in I}\bigcup_{l=0}^6D[i_0,i_1,\ldots,i_n,i_{n+1},\ldots,i_q,l]\]
\[=\bigcup_{(i_n,i_{n+1},\ldots,i_q)\in I}D[i_0,i_1,\ldots,i_n,i_{n+1},\ldots,i_q].\]
So 
\[D=\bigcup\{X\,:\,X\in M\}=\bigcup\{X\,:\,X\in M'\}\cup\bigcup\{X\,:\,X\in M''\}\]
\[=\bigcup\{X\,:\,X\in M'\}\cup\left(
\bigcup_{(i_n,i_{n+1},\ldots,i_q)\in I}D[i_0,i_1,\ldots,i_n,i_{n+1},\ldots,i_q]\right).\]
All sets in these unions are pairwise disjoint and have lengths $\leqslant q$, so by the inductive assumption and 
\eqref{eq:P_0_X}:
\[P_0(D)=\sum_{X\in M'}P_0(X)+\sum_{(i_n,i_{n+1},\ldots,i_q)\in I}P_0\left(D[i_0,i_1,\ldots,i_n,i_{n+1},\ldots,i_q]\right)
\]
\[=\sum_{X\in M'}P_0(X)+\sum_{X\in M''}P_0(X)
=\sum_{X\in M}P_0(X)=\sum\limits_{j=0}^k P_0(D_j).\]

This completes the proof of the inductive step.
\end{proof}

\begin{lemma}
Suppose $A\in\mathcal{F}_0$ is represented as a union of pairwise disjoint planes in two ways:
\[A=\bigcup\limits_{j=0}^n D_j\;\text{ 
and }
\;A=\bigcup\limits_{k=0}^m D'_k.\]

Then \[\sum\limits_{j=0}^n P_0(D_j)=\sum\limits_{k=0}^m P_0(D'_k).\]
\label{lemma:pre_measure}
\end{lemma}
\begin{proof}
Clearly each $D_j\subseteq A,$ so 
\[D_j=D_j\cap A=D_j\cap\bigcup\limits_{k=0}^m D'_k=\bigcup\limits_{k=0}^m(D_j\cap D'_k).\]

By Proposition \ref{prop:planes}.4), each $D_j\cap D'_k$ is $\varnothing$ or a plane; these planes are pairwise disjoint. Since $P_0(\varnothing)=0$, by Lemma \ref{lemma:hard_proof}  we have for any $j=0,1,2,\ldots,n:$
\[
P_0(D_j)=\sum\limits_{k=0}^m P_0(D_j\cap D'_k).\]

Similarly for any $k=0,1,2,\ldots,m:$
\[
P_0(D'_k)=\sum\limits_{j=0}^n P_0(D_j\cap D'_k).
\]

So  \[\sum\limits_{j=0}^n P_0(D_j)
=\sum\limits_{j=0}^n \sum\limits_{k=0}^m P_0(D_j\cap D'_k)
=
\sum\limits_{k=0}^m \sum\limits_{j=0}^nP_0(D_j\cap D'_k)
=
\sum\limits_{k=0}^m P_0(D'_k).\]
\end{proof}

\begin{definition}
We define function $P_0:\mathcal{F}_0\rightarrow [0,1]$ by the following.
\begin{enumerate}
\item When $A=\varnothing$ or $A$ is a plane, $P_0(A)$ has been defined in Definition \ref{def:P_0}.
\medskip
\item By Proposition \ref{prop:union_planes}.2), any non-empty $A\in\mathcal{F}_0$ can be represented as $A=\bigcup\limits_{j=0}^m D_j$, where $D_j$ are pairwise disjoint planes. We define:
\[P_0(A)=\sum\limits_{j=0}^m P_0(D_j).\]
\end{enumerate}
\end{definition}

Due to Lemma \ref{lemma:pre_measure}, this definition is valid.
\medskip

\begin{lemma}
For any $m\in\mathbb{N}$ and distinct $n_{0},n_{1},\dots,n_{m}\in\mathbb{N}$: 
\[P_0\left(\bigcap_{k=0}^mC(n_k,i_k) \right)=\prod_{k=0}^m P_0\big(C(n_k,i_k)\big).\]
\label{lemma:P_0_product}
\end{lemma}

\begin{proof}
We will prove the lemma for this particular case:
\[P_0\big(C(1,k)\cap C(3,l)\big)=P_0\big(C(1,k)\big) \cdot P_0\big(C(3,l)\big).\]
The general case is proven similarly.

By Proposition \ref{prop:hyperplane}.2), we have:
\[C(1,k)\cap C(3,l)=
\Omega\cap C(1,k)\cap\Omega\cap C(3,l)
=\left[\bigcup_{i=0}^6 C(0,i)\right]\cap C(1,k)\cap\left[\bigcup_{j=0}^6 C(2,j)\right]\cap C(3,l)
\]
\[=\bigcup_{i=0}^6 \bigcup_{j=0}^6\big[C(0,i)\cap C(1,k)\cap C(2,j)\cap C(3,l)\big]
=\bigcup_{i,j=0}^6 D[i,k,j,l].\]

\[\text{So }
P_0\big(C(1,k)\cap C(3,l)\big)
=\sum_{i,j=0}^6P_0\left( D[i,k,j,l]\right)
\]
\[=\sum_{i,j=0}^6 P_0\big(C(0,i)\big) \cdot P_0\big(C(1,k)\big)\cdot P_0\big(C(2,j)\big)
\cdot P_0\big(C(3,l)\big)\]
\[=\left[\sum_{i=0}^6 P_0\big(C(0,i)\big)\right]\cdot P_0\big(C(1,k)\big)\cdot
\left[\sum_{j=0}^6 P_0\big(C(2,j)\big)\right]\cdot P_0\big(C(3,l)\big)
=P_0\big(C(1,k)\big) \cdot P_0\big(C(3,l)\big),\]
since each sum in square brackets equals 1 by the definition of $P_0$.
\end{proof}

\subsubsection{Defining $P$ on $\mathcal{F}$} 
We constructed a pre-measure $P_0$ on algebra $\mathcal{F}_0$ over $\Omega$. By Hahn–Kolmogorov extension theorem $P_0$ can be extended from $\mathcal{F}_0$ to probability measure $P$ on $\mathcal{F}$. That means: 
\[(\Omega,\mathcal{F},P)\text{  is a probability space,}\] 
\[\text{and for any }A\in\mathcal{F}_0,\;P_0(A)=P(A).\]
\medskip

Lemma \ref{lemma:P_0_product} implies the following corollary.

\begin{cor}
For any $m\in\mathbb{N}$ and distinct $n_{0},n_{1},\dots,n_{m}\in\mathbb{N}$: 
\[P\left(\bigcap_{k=0}^mC(n_k,i_k) \right)=\prod_{k=0}^m P\big(C(n_k,i_k)\big).\]
\label{corollary:P_product}
\end{cor}

\section{Motion of a particle in discrete space and time}
In next two sections we will derive analogues of two Newton's laws of motion from probabilistic characteristics of particle movement. 

\subsection{Analogue of variance for a random vector}
Here we introduce some preliminary concepts.
\begin{notation}
For a random 3-dimensional vector $\boldsymbol{X}=(X_1,X_2,X_3)$ we denote:
\[\boldsymbol{Tr(X)}=Var(X_1)+Var(X_2)+Var(X_3).\]
\end{notation}
Thus, $Tr(\boldsymbol{X})$ is the trace of the covariance matrix $Cov(\boldsymbol{X})$.

\begin{lemma} 
$Tr(\boldsymbol{X})=E\big[(\boldsymbol{X}-E(\boldsymbol{X}),\boldsymbol{X}-E(\boldsymbol{X}))\big]$, the expectation of the inner product of $\boldsymbol{X}-E(\boldsymbol{X})$ by itself.
\label{lemma:var}
\end{lemma} 

\begin{proof}
$(\boldsymbol{X}-E(\boldsymbol{X}),\boldsymbol{X}-E(\boldsymbol{X}))
=\sum\limits_{i=1}^3(X_i-E(X_i))^2$, so 
\[E\big[(\boldsymbol{X}-E(\boldsymbol{X}),\boldsymbol{X}-E(\boldsymbol{X}))\big]=\sum\limits_{i=1}^3E\big[(X_i-E(X_i))^2\big]=\sum\limits_{i=1}^3Var(X_i)= Tr(\boldsymbol{X}).\]
\end{proof}

\subsection{Random walk on a lattice: definitions} 

\begin{notation}
1) We fix a real parameter $\boldsymbol{\tau}>0$. 

$\tau>0$ is interpreted as a unit of time. When $\tau\rightarrow 0$, we get continuous time.
\medskip

2) We fix a real number $\boldsymbol{c}>0$ and denote $\boldsymbol{\varepsilon} =c\tau$.

3) We consider a 3-dimensional \textbf{lattice}:
$$\varepsilon\mathbb{Z}^3= \{(\varepsilon l_1,\varepsilon l_2,\varepsilon l_3)\, :\, l_1,l_2,l_3\in \mathbb{Z}\}.$$

When $\tau\rightarrow 0$, then $\varepsilon\rightarrow0$ and this lattice becomes the continuous 3-dimensional space $\mathbb{R}^3$.
\medskip

4) We will use the following vectors of the standard basis in $\mathbb{R}^3$:
\[\mathbf{e}_1=(1,0,0),\mathbf{e}_2=(0,1,0),\text{ and }\mathbf{e}_3=(0,0,1).\]

5) We fix the following 7 vectors:
\begin{itemize}
\item $\mathbf{s}(0)=(0,0,0)$;
\medskip
\item $\mathbf{s}(1)=\varepsilon\mathbf{e}_1$;$\quad\mathbf{s}(2)=\varepsilon\mathbf{e}_2$;$\quad\mathbf{s}(3)=\varepsilon\mathbf{e}_3$;
\medskip
\item $\mathbf{s}(4)=-\varepsilon\mathbf{e}_1$;$\quad\mathbf{s}(5)=-\varepsilon\mathbf{e}_2$;$\quad\mathbf{s}(6)=-\varepsilon\mathbf{e}_3$.
\end{itemize}
\medskip

The vectors $\mathbf{s}(1),\mathbf{s}(2),\mathbf{s}(3)$ represent
 steps in the directions $\mathbf{e}_1,\mathbf{e}_2,\mathbf{e}_3$, respectively, and the vectors $\mathbf{s}(4),\mathbf{s}(5),\mathbf{s}(6)$ represent
 steps in the opposite directions.
\end{notation}

\begin{lemma} 
\label{var1}
Suppose a random vector $\boldsymbol{X}$ has a discrete distribution with values $\mathbf{s}(0),$ $\mathbf{s}(1),\ldots,$ $\mathbf{s}(6)$. Then $Tr(\boldsymbol{X})\leqslant \varepsilon^2$.
\label{lemma:Tr}
\end{lemma}
\begin{proof}
Denote $\boldsymbol{M}=E(\boldsymbol{X})$. Then 
\begin{equation}
\boldsymbol{M}=\sum_{j=0}^6\mathbf{s}(j)P(\boldsymbol{X}=\mathbf{s}(j)).
\label{eq:expectation_1}
\end{equation}
\[(\boldsymbol{X},\boldsymbol{M})=\sum\limits_{i=1}^3M_i\cdot X_i\;\text{ and }\;E(\boldsymbol{X},\boldsymbol{M})=\sum\limits_{i=1}^3E(M_iX_i)=\sum\limits_{i=1}^3\sum\limits_{j=0}^6M_is_{i}(j)P(\boldsymbol{X}=\mathbf{s}(j))\]
\[=\sum\limits_{i=1}^3M_i\sum\limits_{j=0}^6s_{i}(j)P(\boldsymbol{X}=\mathbf{s}(j))
=\sum\limits_{i=1}^3M_i\cdot M_i \text{ by \eqref{eq:expectation_1}. So}\]
\begin{equation}
E(\boldsymbol{X},\boldsymbol{M})=(\boldsymbol{M},\boldsymbol{M}).
\label{eq:expectation_2}
\end{equation}

By Lemma \ref{lemma:var}, $Tr(\boldsymbol{X})=E\big[(\boldsymbol{X}-\boldsymbol{M},\boldsymbol{X}-\boldsymbol{M})\big]
=E(\boldsymbol{X},\boldsymbol{X})- 2E(\boldsymbol{X},\boldsymbol{M})+(\boldsymbol{M},\boldsymbol{M})$, since $\boldsymbol{M}$ is a constant vector. So by \eqref{eq:expectation_2},
\[
Tr(\boldsymbol{X})=
E(\boldsymbol{X},\boldsymbol{X})-(\boldsymbol{M},\boldsymbol{M}).\]

To complete the proof, it is sufficient to show: $E(\boldsymbol{X},\boldsymbol{X})\leqslant \varepsilon^2.$

\[E(X_1^2)=\sum_{j=0}^6s_1^2(j)P(\boldsymbol{X}=\mathbf{s}(j))
=s_1^2(1)P(\boldsymbol{X}=\mathbf{s}(1))+s_1^2(4)P(\boldsymbol{X}=\mathbf{s}(4))\]
\[=\varepsilon^2 P(\boldsymbol{X}=\mathbf{s}(1))+\varepsilon^2P(\boldsymbol{X}=\mathbf{s}(4))
\text{ by the definition of }
\mathbf{s}(j), \;j=0,1,2\ldots,6.\] So
\begin{equation}
E(X_1^2)=\varepsilon^2 \big[P(\boldsymbol{X}=\mathbf{s}(1))+P(\boldsymbol{X}=\mathbf{s}(4))\big]
\label{eq:X_square_1}.
\end{equation}
Similarly,
\begin{equation}
E(X_2^2)=\varepsilon^2 \big[P(\boldsymbol{X}=\mathbf{s}(2))+P(\boldsymbol{X}=\mathbf{s}(5))\big]
\label{eq:X_square_2},
\end{equation}

\begin{equation}
E(X_3^2)=\varepsilon^2 \big[P(\boldsymbol{X}=\mathbf{s}(3))+P(\boldsymbol{X}=\mathbf{s}(6))\big]
\label{eq:X_square_3}.
\end{equation}

Adding \eqref{eq:X_square_1} - \eqref{eq:X_square_3}, we get:
\[E(\boldsymbol{X},\boldsymbol{X})
=\sum\limits_{i=1}^3
E(X_i^2)
=\varepsilon^2 \sum\limits_{j=1}^6P(\boldsymbol{X}=\mathbf{s}(j))
\leqslant \varepsilon^2.\]

\end{proof}

\begin{definition}\label{def:walk}
1) We define 3-dimensional random vectors $\mathbf{S}_n$ on $(\Omega,\mathcal{F},P)$ by the following:
\[\mathbf{S}_n(\omega)=\mathbf{s}(\omega_n),\;n\in\mathbb{N}.\]

2) We call a \textbf{particle walk} the sequence $\mathbf{R}_0, \mathbf{R}_1,\mathbf{R}_2,\mathbf{R}_3,\ldots$ of random vectors  defined by:
\[\mathbf{R}_0(\omega)=\mathbf{s}(0)=(0,0,0)\text{ for any }\omega\in\Omega,\]
\[\text{for }n\geqslant1,\mathbf{R}_n=\mathbf{S}_0+\mathbf{S}_1+\ldots+\mathbf{S}_{n-1}.\]
Thus, $\mathbf{R}_0, \mathbf{R}_1,\mathbf{R}_2,\mathbf{R}_3,\ldots$ is a random walk on the probability space $(\Omega,\mathcal{F},P)$. 
\medskip

$\mathbf{R}_n$ is interpreted as the \textbf{position} of the particle at time $n\tau$. 

We call $\mathbf{S}_n$ a \textbf{step} of the particle walk; it is interpreted as displacement of the particle over the time interval $[n\tau,n\tau+\tau]$. 
\end{definition}

The particle walk $\mathbf{R}_0, \mathbf{R}_1,\mathbf{R}_2,\mathbf{R}_3,\ldots$ models movement of a particle on the lattice $\varepsilon \mathbb{Z}^3$: at each moment $n\tau$ in time the particle rests or moves randomly in one of the 6 directions: $\pm\mathbf{e}_1,\pm\mathbf{e}_2,\pm\mathbf{e}_3$. 
For each $\omega\in\Omega$ the sequence $(\mathbf{R}_0(\omega),\mathbf{R}_1(\omega),\mathbf{R}_2(\omega),\ldots)$ represents a possible path/trajectory of the particle.

\begin{proposition} \label{indep}
1) $P(\mathbf{S}_n=\mathbf{s}(j))=p_n(j)$ for any $n\in\mathbb{N}$, $j\in\mathbb{N}_7$.
\medskip

2) The random vectors $\mathbf{S}_0, \mathbf{S}_1,\mathbf{S}_2,\mathbf{S}_3,\ldots$ are independent.
\label{lemma:steps}
\end{proposition}

The first part of this proposition means: $p_n(1)$, $p_n(2)$, $p_n(3)$ are the probabilities that at time $n\tau$ the particle moves by one step in the directions $\mathbf{e}_1$, $\mathbf{e}_2$, $\mathbf{e}_3$, respectively, and $p_n(4)$, $p_n(5)$, $p_n(6)$ are the same probabilities for the directions $-\mathbf{e}_1$, $-\mathbf{e}_2$, $-\mathbf{e}_3$. 

$p_n(0)$ is the probability that at time $n\tau$ the particle rests.

\begin{proof}
$\{\mathbf{S}_n=\mathbf{s}(j)\}=\{\omega\in\Omega\;:\;\mathbf{s}(\omega_{n})=\mathbf{s}(j)\}=\{\omega\in\Omega\;:\;\omega_n=j\}=C(n,j)$. Thus,
\begin{equation}
\{\mathbf{S}_n=\mathbf{s}(j)\}=C(n,j).
\label{eq:C_nj}
\end{equation}

1) $P(\mathbf{S}_n=\mathbf{s}(j))=P(C(n,j))=p_n(j)$ by the definition of probability $P$.
\medskip

2) For any $k\geqslant1$, $j_0,j_1,\ldots,j_k\in\mathbb{N}_7$, and distinct $n_1,n_2,\ldots,n_k\in\mathbb{N}$:
\[\{\mathbf{S}_{n_1}=\mathbf{s}(j_1),\ldots,\mathbf{S}_{n_k}=\mathbf{s}(j_k)\}
=\{\omega\in\Omega\;:\;\mathbf{s}(\omega_{n_1})=\mathbf{s}(j_1),\ldots,\mathbf{s}(\omega_{n_k})=\mathbf{s}(j_k)\}\]
\[=\{\omega\in\Omega\;:\;\omega_{n_1}=j_1,\ldots,\omega_{n_k}=j_k\}=C(n_1,j_1)\cap\ldots\cap C(n_k,j_k).\]
\[\text{So }
P\big(\mathbf{S}_{n_1}=\mathbf{s}(j_1),\ldots,\mathbf{S}_{n_k}=\mathbf{s}(j_k)\big) =P\big(C(n_1,j_1)\cap\ldots\cap C(n_k,j_k)\big)\]
\[=P\big(C(n_1,j_1)\big)\cdot\ldots\cdot P\big(C(n_k,j_k)\big)
=P\big(\mathbf{S}_{n_1}=\mathbf{s}(j_1)\big)\cdot\ldots\cdot P\big(\mathbf{S}_{n_k}=\mathbf{s}(j_k)\big).\]

Here we used Corollary \ref{corollary:P_product} and \eqref{eq:C_nj}. 
\end{proof}

\begin{definition}
1) \textbf{Velocity} of the particle at time $n\tau$ is $\mathbf{v}(n)= \dfrac{E(\mathbf{S}_n)}{\tau}$.
\medskip

2) \textbf{Acceleration} of the particle at time $n\tau$ is $\mathbf{a}(n)=
\dfrac{\mathbf{v}(n+1)-\mathbf{v}(n)}{\tau}$.
\label{def:velocity}
\end{definition}

We use this definition for acceleration because it is the change of velocity over unit of time.

\begin{lemma}
$\mathbf{v}(n)=\dfrac{\varepsilon}{\tau}\sum\limits_{i=1}^3\big[p_n(i) -p_n(i+3)\big]\mathbf{e}_i.$
\label{lemma:velocity}
\end{lemma}
\begin{proof}
Using Proposition \ref{lemma:steps}.1), we get:
\[\mathbf{v}(n)=
\dfrac{1}{\tau}E(\mathbf{S}_n)= \dfrac{1}{\tau}\sum_{j=0}^6\mathbf{s}(j) P\big(\mathbf{S}_n =\mathbf{s}(j)\big) =\dfrac{1}{\tau}\sum_{j=1}^6\mathbf{s}(j) p_n(j) \]
\[=\dfrac{\varepsilon}{\tau}\sum_{i=1}^3 p_n(i)\mathbf{e}_i-\dfrac{\varepsilon}{\tau}\sum_{i=1}^3 p_n(i+3)\mathbf{e}_i
=\dfrac{\varepsilon}{\tau}\sum\limits_{i=1}^3\big[p_n(i) -p_n(i+3)\big]\mathbf{e}_i.\]
\end{proof}

\section{Motion with constant velocity}

Here we state an axiom about probabilities of particle motion under zero resultant force and derive an analogue of Newton's first law of motion. 

\textbf{Axiom 1.}
\textit{For any $n\in\mathbb{N}$, if no force acts on the particle during time $[n\tau,(n+1)\tau]$, then $p_{n+1}(j)=p_n(j)$ for any $j\in\mathbb{N}_7$.}
\medskip

In other words, this axiom states that with no force acting on the particle the probabilities of its movement in each direction stay constant.
\medskip

Instead of Axiom 1 we will use its weaker form - Axiom 1*.

\textbf{Axiom 1*.}
\textit{For any $n\in\mathbb{N}$, if the resultant force on the particle during time interval $[n\tau,(n+1)\tau]$ equals zero, then $p_{n+1}(i) -p_{n+1}(i+3)=p_n(i)-p_n(i+3)$ for each $i\in\{1,2,3\}$.}
\medskip

We refer to a zero resultant force in this axiom because it is well-known that motion under zero resultant force is the same as under no force. Clearly, Axiom 1* follows from Axiom 1. Due to Lemma \ref{lemma:velocity}, Axiom 1* implies that with zero resultant force the particle's velocity is constant during the interval.

\begin{theorem} 
Suppose the resultant force on the particle over time interval  $[0,N\tau]$ equals zero ($N\in\mathbb{N}$). Then the following hold.

1) The particle moves with constant velocity during the time interval $[0,N\tau]$, that is the velocity 
$\mathbf{v}=\mathbf{v}(n)$ is the same for all $n\leqslant N$.
\medskip

2) For any $n\leqslant N$: 
\[E(\mathbf{R}_n) =n\tau\mathbf{v}.\] 

3) Suppose $N\rightarrow \infty$ in such a way that $N\tau=Const$ (so the time interval $[0,N\tau]$ is fixed). Then 
\[Tr(\mathbf{R}_N)\rightarrow 0.\] 
\label{theorem:Newton_1}
\end{theorem}

This is interpretation of Theorem \ref{theorem:Newton_1}: 
when no force acts on the particle during a large time interval, then: 
\begin{itemize}
\item the particle's motion follows approximately the Newton's first law of motion during this time interval: the motion is uniform and is along a straight line (parts 1) and 2) of the theorem);
\item the deviation from this law gets very small as the time unit $\tau$ approaches zero, i.e.  time becomes continuous (part 3) of the theorem). The time unit $\tau=\dfrac{Const}{N}$, so $\tau\rightarrow 0$ as $N\rightarrow \infty$.
\end{itemize}

\begin{proof} 
1) By Lemma \ref{lemma:velocity}, 
$\mathbf{v}(n)=\dfrac{\varepsilon}{\tau}\sum\limits_{i=1}^3\big[p_n(i) -p_n(i+3)\big]\mathbf{e}_i.$
\medskip

Denote $a_i=p_n(i)-p_n(i+3)$. 
By Axiom 1*, $a_i$ does not depend on $n\in[0,N]$ for any $i\in\{1,2,3\}$. 
So $\mathbf{v}(n)=\dfrac{\varepsilon}{\tau}\sum\limits_{i=1}^3a_i\mathbf{e}_i$ is the same for all $n\leqslant N$.
\medskip

2) Suppose $n\leqslant N$. By the definition of velocity, $E(\mathbf{S}_n)=\tau\mathbf{v}(n)= \tau\mathbf{v}$ by part 1).
\medskip

$\mathbf{R}_n=\mathbf{S}_0+\mathbf{S}_1+\ldots+\mathbf{S}_{n-1}$, so
\[E(\mathbf{R}_n)=
\sum_{i=0}^{n-1}
E(\mathbf{S}_i)
=\sum_{i=0}^{n-1}
\tau\mathbf{v}
=n\tau\mathbf{v}.\]
\medskip

3) Suppose $N\rightarrow \infty$ in such a way that $N\tau=C,$ where $C$ is a constant.
\medskip

By Lemma \ref{lemma:Tr}, $Tr(\boldsymbol{S}_n)\leqslant \varepsilon^2$
for any $n\in\mathbb{N}$.
By Proposition \ref{lemma:steps}.2), the random vectors $\mathbf{S}_0, \mathbf{S}_1,\mathbf{S}_2,\mathbf{S}_3,\ldots$ are independent. So
\[Tr(\boldsymbol{R}_n)=
\sum_{i=0}^{n-1}
Tr(\boldsymbol{S}_i)\leqslant n\varepsilon^2\leqslant N\varepsilon^2.\]

Since $\varepsilon=c\tau=\dfrac{cC}{N}$, we have
\[Tr(\boldsymbol{R}_n)\leqslant N\varepsilon^2\leqslant N 
\left(\dfrac{cC}{N}\right)^2=\dfrac{(cC)^2}{N}\rightarrow 0\text{ as }N\rightarrow \infty.\]
\end{proof}

\section{Accelerated motion}

In this section we consider motion of the particle under a non-zero force.

\begin{proposition}[Recurrence relation]
Motion of the particle under any force satisfies the following recurrence relation:
\[E(\mathbf{R}_{n+2})=2E(\mathbf{R}_{n+1}) -E(\mathbf{R}_n)+ \tau^2\mathbf{a}(n).\]
\end{proposition}

This is an analogue of the approximate formula for acceleration:
\[\ddot{\mathbf{r}}\approx
\dfrac{\mathbf{r}(t+\tau)-2\mathbf{r}(t)+\mathbf{r}(t-\tau)}{\tau^2}
,\]
where $\mathbf{r}(t)$ is the position function; this formula is easily derived using Taylor's formula.
\begin{proof}
By the definitions of velocity and $\mathbf{R}_n$ we have:
\[\mathbf{v}(n)=
\frac{1}{\tau}
E(\mathbf{S}_n) =\frac{1}{\tau}\big[E(\mathbf{R}_{n+1})-E(\mathbf{R}_n)\big],\]
\[\mathbf{v}(n+1)=
\frac{1}{\tau}\big[E(\mathbf{R}_{n+2})-E(\mathbf{R}_{n+1})\big].\]

Subtracting these equalities, we get:
\[\mathbf{a}(n)=\frac{1}{\tau}\big[\mathbf{v}(n+1)-\mathbf{v}(n)\big]=  \frac{1}{\tau^2}\big[E(\mathbf{R}_{n+2})- 2E(\mathbf{R}_{n+1})+E(\mathbf{R}_n)\big]\text{ and}\]
\[E(\mathbf{R}_{n+2})- 2E(\mathbf{R}_{n+1})+E(\mathbf{R}_n)=\tau^2\mathbf{a}(n).\]
\end{proof}

We derive next results from the following axiom about probabilities of the particle movement.

\textbf{Axiom 2.}
\textit{Let $\boldsymbol{F}(n)= \big(F_1 (n),F_2 (n),F_3 (n)\big)$ be a resultant force acting on the particle at time $n\tau$, and let each $F_i(n)$ be a difference of two forces:
\[F_i(n)=f_n(i)-f_n(i+3),\]
where the force $f_n(i)$ acts in the direction of $\mathbf{e}_i$ and $f_n(i+3)$ - in the opposite direction.}

\textit{We assume that there is a constant $\gamma>0$ such that for any $j=1,2,3,4,5,6$ and $n\leqslant N$:}
\[p_{n+1}(j)-p_n(j)=\gamma f_n(j).\]

Axiom 2 means that during each time unit the change in the probability of moving in direction $j$ is proportional to the force acting in this direction. In short: changes in motion probabilities are proportional to the forces acting on the particle.

\begin{theorem}
The resultant force acting on the particle is proportional to its acceleration, i.e. there is a constant $\beta>0$ such that for any $n\leqslant N$:
\begin{equation}
\boldsymbol{F}(n)=\beta\mathbf{a}(n).
\label{eq:Newton_2}
\end{equation}
\end{theorem}

This theorem is a weaker analogue of Newton's second law of motion:
$\mathbf{F}=m\mathbf{a}$. From \eqref{eq:Newton_2} we can define the \textbf{mass} of the particle as the coefficient of proportionality $\beta$.

\begin{proof}
For any $n\leqslant N$ and $i\in\{1,2,3\}$ by Axiom 2 we have: \[F_i(n)= f_n(i)-f_{n}(i+3)=\frac{1}{\gamma}\big[p_{n+1}(i)-p_n(i)\big]
-\frac{1}{\gamma}\big[p_{n+1}(i+3)-p_n(i+3)\big]\]
\[=\frac{1}{\gamma}\big[p_{n+1}(i)-p_{n+1}(i+3)\big]
-\frac{1}{\gamma}\big[p_n(i)-p_n(i+3)\big].\]

So 
\[\boldsymbol{F}(n)
=\sum_{i=1}^3F_i(n)\mathbf{e}_i
=\frac{1}{\gamma}\sum_{i=1}^3\big[p_{n+1}(i)-p_{n+1}(i+3)\big]
\mathbf{e}_i
-\frac{1}{\gamma}\sum_{i=1}^3\big[p_n(i)-p_n(i+3)\big]\mathbf{e}_i\]
\[=\frac{1}{\gamma}\cdot\frac{\tau}{\varepsilon}
\mathbf{v}(n+1)
-\frac{1}{\gamma}\cdot\frac{\tau}{\varepsilon}
\mathbf{v}(n)=\frac{\tau}{\varepsilon\gamma}[\mathbf{v}(n+1)-\mathbf{v}(n)]
=\frac{\tau^2}{\varepsilon\gamma}
\cdot\frac{\mathbf{v}(n+1)-\mathbf{v}(n)}{\tau}
=\frac{\tau^2}{\varepsilon\gamma}
\,\mathbf{a}(n)\]
by Lemma \ref{lemma:velocity} and the definition of acceleration $\mathbf{a}(n)$.

Thus, $\boldsymbol{F}(n)=\beta\mathbf{a}(n),$ where $\beta=\dfrac{\tau^2}{\varepsilon\gamma}>0.$
\end{proof}
 
\begin{proposition}[Motion under constant forces] Suppose in each direction the acting force is constant, that is for any $j=1,2,\ldots,6$, $f(j)=f_n(j)$ does not depend on $n$ $(n\leqslant N)$. Then there exist constant vectors $\boldsymbol{a}$ and $\boldsymbol{b}$ such that for any $n\leqslant N$:
\begin{equation}
E(\mathbf{R}_n)=\varepsilon n\left(\frac{n-1}{2}\boldsymbol{a}+\boldsymbol{b}\right).
\label{eq:parabola_3D}
\end{equation}
\label{prop:parabola}
\end{proposition}

Proposition \ref{prop:parabola} means that when the particle moves under constant forces, its trajectory is 
approximately a parabola in 3-D. This is because \eqref {eq:parabola_3D} is a parametric equation of a parabola (with quadratic function of time $n\tau$). 

\begin{proof}
By Axiom 2 we have
\begin{equation}
p_{n+1}(j)=p_n(j)+\gamma f(j).
\label{eq:axiom_2}
\end{equation}

Denote 
\begin{itemize}
\item $a_i=\gamma[f(i)-f(i+3)]$, $\;b_i=p_0(i)-p_0(i+3)$,
\medskip
\item $\boldsymbol{a}=(a_1,a_2,a_3)$, $\;\boldsymbol{b}=(b_1,b_2,b_3)$.
\end{itemize} 

We prove the following formula by induction on $n$:
\begin{equation}
p_n(i)-p_n(i+3)=a_in+b_i.
\label{eq:parabola}
\end{equation}

If $n=0$, then the formula obviously holds.

Assume it holds for $n$. By \eqref{eq:axiom_2}, 
\[p_{n+1}(i)-p_{n+1}(i+3)=
[p_n(i)+\gamma f(i)]-
[p_n(i+3)+\gamma f(i+3)]
=p_n(i)-p_n(i+3)+
\gamma [f(i)-f(i+3)]
\]
\[=(a_in+b_i)+a_i=a_i(n+1)+b_i\text{ 
by the inductive assumption}.\]
This completes the proof of \eqref{eq:parabola}.

By Lemma \ref{lemma:velocity} and \eqref{eq:parabola}: 
\begin{equation}
E(\mathbf{S}_n)=\tau\mathbf{v}_n=\varepsilon\sum\limits_{i=1}^3\big[p_n(i) -p_n(i+3)\big]\mathbf{e}_i
=\varepsilon\sum\limits_{i=1}^3(a_in+b_i)\mathbf{e}_i.
\label{eq:parabola_extra}
\end{equation}
\[E(\mathbf{R}_n)=E(\mathbf{S}_0)+E(\mathbf{S}_1)+\ldots+E(\mathbf{S}_{n-1}).\]

By \eqref{eq:parabola_extra} the first coordinate of $E(\mathbf{S}_n)$ is $E(\mathbf{S}_n)_1=\varepsilon(a_1n+b_1)$, so the first coordinate of $E(\mathbf{R}_n)$ equals
\[E(\mathbf{R}_n)_1
=\varepsilon b_1 +\varepsilon(a_1+b_1)+\varepsilon(2a_1+b_1)+\ldots+\varepsilon[a_1(n-1)+b_1]=
\varepsilon[nb_1+ a_1(1+2+\ldots+(n-1))]
,\text{ and}\]
\begin{equation}
E(\mathbf{R}_n)_1=\varepsilon n\left[
\frac{n-1}{2}a_1+b_1\right].
\label{eq:coord_1}
\end{equation}

Similarly, for the second and third coordinates of $E(\mathbf{R}_n)$ we have:
\begin{equation}
E(\mathbf{R}_n)_2=\varepsilon n\left[
\frac{n-1}{2}a_2+b_2\right],
\label{eq:coord_2}
\end{equation}
\begin{equation}
E(\mathbf{R}_n)_3=\varepsilon n\left[
\frac{n-1}{2}a_3+b_3\right].
\label{eq:coord_3}
\end{equation}

Combining \eqref{eq:coord_1} - \eqref{eq:coord_3}, we get:
\[E(\mathbf{R}_n)=\varepsilon n\left(\frac{n-1}{2}\boldsymbol{a}+\boldsymbol{b}\right).\]
\end{proof}

\section{Conclusion}
In this paper we use a rigorous approach to construct a probability space for trajectories of a particle on three-dimensional lattice with small cells. We create a model of particle motion as a random walk on this probability space. Based on two natural assumptions about probabilistic characteristics of particle movement, we formally derive analogues of Newton's first and second laws of motion, and some related facts.

As we described in the Introduction, a similar approach can potentially be applied to resolve some inconsistencies in thermodynamics theory. We plan to continue our research in this direction by generalizing the probabilistic model from this paper to many-particle systems. We aim to use such a model to formally derive laws of thermodynamics in a consistent way. 

\section*{Ethics}
This is a mathematical article; no ethical issues can arise after its publication.

\bibliographystyle{plain}
\bibliography{ilias}

\end{document}